\documentclass[12pt]{article}

\usepackage{amsfonts, amsmath, amsthm}
\usepackage{geometry}
\usepackage{hyperref}
\usepackage{graphicx}
\usepackage{color}

\newcounter{dummy} \numberwithin{dummy}{section}
\newtheorem{theorem}[dummy]{Theorem}
\newtheorem{proposition}[dummy]{Proposition}
\newtheorem{lemma}[dummy]{Lemma}
\newtheorem{corollary}[dummy]{Corollary}

\newtheorem{remark}{Remark}[section]

\def\({\left(}
\def\){\right)}

\def\R{\mathbb{R}}
\def\Z{\mathbb{Z}}

\def\E{\mathbf{E}}
\def\P{\mathbf{P}}

\def\W{\mathcal{W}}
\def\Siginv{{\Sigma^{-1}}}
\def\sigmamin{\sigma_{\text{min}}}
\def\dci{\Delta_{CI}}
\newcommand{\KL}[2]{\mathcal{D}(#1 \,\|\, #2)}

\DeclareMathOperator{\Vol}{Vol}

\usepackage{bigints}

\begin{document}

\title{A high-dimensional CLT in $\W_2$ distance with near optimal
  convergence rate} \author{Alex Zhai \\ Stanford University}
\maketitle

\begin{abstract}
  Let $X_1, \ldots , X_n$ be i.i.d. random vectors in $\R^d$ with
  $\|X_1\| \le \beta$. Then, we show that
  \[ \frac{1}{\sqrt{n}}\(X_1 + \ldots + X_n\) \]
  converges to a Gaussian in quadratic transportation (also known
    as ``Kantorovich'' or ``Wasserstein'') distance at a rate of $O
  \( \frac{\sqrt{d} \beta \log n}{\sqrt{n}} \)$, improving a result of
  Valiant and Valiant. The main feature of our theorem is that the
  rate of convergence is within $\log n$ of optimal for $n, d
  \rightarrow \infty$.
\end{abstract}

\section{Introduction}

The central limit theorem states that if $X_1, X_2, \ldots , X_n$ are
independent and identically distributed random variables, then $S_n =
\frac{1}{\sqrt{n}} \sum_{i = 1}^n X_i$ is approximately Gaussian. It
is well-known that by various metrics the distance from Gaussian
decays at a rate of $n^{-1/2}$; for example, the celebrated
Berry-Esseen bound states that $| \P(S \le t) - \P(Z \le t) | = O\(
n^{-1/2} \E |X_i|^3 \)$. Moreover, this bound is optimal to within a
constant.

The same principle holds if we allow the $X_i$ to be $\R^d$-valued,
and an extensive literature was developed, tracing back at least to
the 1940's \cite{B45} (see also \cite{B77} and references therein),
around establishing multivariate central limit theorems with good
convergence rates. One new consideration that arises in the
multivariate setting is that the convergence rate is expressed in
terms of not only $n$ but also the dimension $d$. This dependence on
$d$, and in particular when $d$ is growing with $n$, was studied by
Nagaev \cite{N76}, Senatov \cite{S80}, G\"otze \cite{G91}, Bentkus
\cite{B03}, and Chen and Fang \cite{CF11}, among others. These works
focus on convergence in probabilities of convex sets, which we will
call convergence in convex-indicator (CI) distance.

In addition to being a natural question, obtaining good dependence on
dimension has recently been of interest in various
applications. Bubeck and Ganguly \cite{BG15} prove a central limit
theorem for Wishart matrices (relevant to random geometric graphs, see
also \cite{BDER16}), and Chernozhukov, Chetverikov, and Kato
\cite{CCK13} prove a central limit theorem for maxima of sums of
independent random vectors (with applications in high-dimensional
statistical inference). Another relevant work is that of Valiant and
Valiant \cite{VV10a}\footnote{See \cite{VV10b} for the full version.},
who prove central limit theorems for transportation distance and
generalized multinomial distributions and use them to construct lower
bounds for estimating entropy.

In this paper, we prove a multivariate central limit theorem for
quadratic transportation distance whose rate of convergence is within
$\log n$ of optimal in both the number of summands $n$ and the
dimension $d$, improving the result of Valiant and Valiant
\cite{VV10a}. To our knowledge, this is the first general multivariate
central limit theorem whose convergence rate is optimal to within
logarithmic factors in both $n$ and $d$, albeit not for the CI metric
that is most commonly studied in the literature.\footnote{It should be
  noted that the bounds obtained by Bubeck and Ganguly \cite{BG15} are
  also optimal to within logarithmic factors, but they are specific to
  Wishart matrices. We mention also the work of Bentkus and G\"otze
  \cite{BG96}, which obtains optimal bounds for quadratic forms under
  certain somewhat specialized assumptions.} Additionally, we believe
that the method of proof based on Talagrand's transportation
inequality, described in Section \ref{subsec:proof-idea}, is of
independent interest. We also note that in certain regimes our result
implies stronger bounds in CI distance than what is known in the
existing literature, as elaborated in Section
\ref{subsec:ci-comparison}.

To state the result, recall that for two probability measures $\mu$
and $\nu$ on $\R^d$ and a number $p \ge 1$, the \emph{$L^p$
  transportation distance}\footnote{Other names appearing in the
  literature include ``Monge-Kantorovich distance'', ``Kantorovich
  distance'', and ``Wasserstein distance''. We refer to \cite{V13} for
  a historical discussion of the concept.}  $\W_p(\mu, \nu)$ is
defined to be
\[ \W_p(\mu, \nu) = \( \inf_{\gamma \in \Gamma(\mu, \nu)} \int \|x - y\|^p \,d\gamma(x, y) \)^{\frac{1}{p}}, \]
where $\Gamma(\mu, \nu)$ is the space of all probability measures on
$\R^d \times \R^d$ with $\mu$ and $\nu$ as marginals. In other words,
$\W_p(\mu, \nu)$ measures how closely $\mu$ and $\nu$ may be
coupled. If $X$ and $Y$ are random variables with distributions $\mu$
and $\nu$, respectively, we will also write
\[ \W_p(X, Y) = \W_p(\mu, \nu). \]
Our main result is the following theorem concerning the $L^2$ (or
  ``quadratic'') transportation distance.
\begin{theorem} \label{thm:main}
  Let $X_1, \ldots , X_n$ be independent random vectors with mean
  zero, covariance $\Sigma$, and $\|X_i\| \le \beta$ almost surely for
  each $i$. Let $S_n = \frac{1}{\sqrt{n}} \sum_{i = 1}^n X_i$, and let
  $Z$ be a Gaussian with covariance $\Sigma$. Then,
  \[ \W_2\( S_n, Z \) \le \frac{5 \sqrt{d} \beta (1 + \log n)}{\sqrt{n}}. \]
\end{theorem}

This bound improves by a factor of $\sqrt{d}$ the result of Valiant
and Valiant \cite{VV10a}, who obtain under the same assumptions a $O
\( \frac{d \beta \log n}{\sqrt{n}} \)$ bound for $\W_1$ distance. In
fact, Theorem \ref{thm:main} is within a $\log n$ factor of optimal,
in the sense that one cannot have a convergence rate faster than
$O(\sqrt{d} \beta / \sqrt{n})$, as shown by the following proposition.
\begin{proposition} \label{prop:lower-bound}
  Let $(X_i)_{i=1}^n$, $S_n$, $Z$, and $\beta$ be as in Theorem
  \ref{thm:main}. Suppose further that the $X_i$ take values in the
  lattice $\beta \Z^d$. Then,
  \[ \liminf_{n \rightarrow \infty} \sqrt{n} \W_2(S_n, Z) \ge \frac{\sqrt{d} \beta}{4}. \]
\end{proposition}
The proof is routine and is given in Appendix
\ref{subsec:lower-bound-proof}; it is based on the fact that a typical
point in $\R^d$ will be a distance $O(\sqrt{d} \beta)$ from the
closest point in $\beta \Z^d$.

Several other works in the literature have studied central limit
theorems for $\W_p$ distance. In the multivariate setting, the recent
work of Bonis \cite{B15} proves a $O(1/\sqrt{n})$ convergence rate for
$\W_2$ distance under the assumption $\E \|X_1\|^4 < \infty$. However,
Bonis' result does not have an explicit dependence on the dimension,
which is the main point of this paper.

We mention also the work of Rio (see \cite{R09}, \cite{R11}), who
analyzed for the one-dimensional setting convergence in $\W_p$
distance under various moment assumptions. For $\W_2$, he proves a
$O(1/\sqrt{n})$ convergence rate under the assumption of finite fourth
moments; we refer the reader to \cite{R09} for statements about other
values of $p$. An alternative proof of Rio's result for $\W_2$ was
given by Bobkov \cite{B13} (see also \cite{BCG14}). We note that
Talagrand's transportation inequality also makes an appearance in
\cite{B13}, but the way it is used is substantially different from the
approach of this paper.

The above literature leads us to believe that Theorem \ref{thm:main}
can be improved to remove the $\log n$ factor (this was also
conjectured in \cite{VV10a}). We remark that the extra $\log n$ factor
in our proof comes from a harmonic series arising from repeated
applications of Lemma \ref{lem:increment} below.

\subsection{Comparison with convex-indicator bounds} \label{subsec:ci-comparison}

For two measures $\mu$ and $\nu$ on $\R^d$, we define the
convex-indicator (CI) distance $\dci$ by
\[ \dci(\mu, \nu) = \sup_{A \subset \R^d \text{ convex}} |\mu(A) - \nu(A)|, \]
and as with $\W_p$ distance, we will write $\dci(X, Y) = \dci(\mu,
\nu)$ if $X$ has distribution $\mu$ and $Y$ has distribution $\nu$. As
mentioned earlier, CI distance is perhaps the most widely studied
metric in the high-dimensional central limit theorem literature (see
e.g. \cite{S68}, \cite{N76}, \cite{S80}, \cite{G91}, \cite{BH10},
\cite{B03}). The best convergence rate seems to be due to Bentkus
\cite{B03}. For simplicity, we state his theorem in the i.i.d. case
(the original paper contains a somewhat more general formulation).

\begin{theorem}[Bentkus, i.i.d. case of Theorem 1.1 in \cite{B03}] \label{thm:bentkus}
  Let $X_1, \ldots , X_n$ be i.i.d. $\R^d$-valued random variables
  with mean zero, identity covariance, and $\( \E \|X_1\|^3
  \)^\frac{1}{3} = \beta_3$. Let $S_n = \frac{1}{\sqrt{n}} \sum_{i =
    1}^n X_i$, and let $Z$ be a standard Gaussian. Then, there is a
  constant $C$ such that
  \[ \dci(S_n, Z) \le \frac{C d^{1/4} \beta_3^3}{\sqrt{n}}. \]
\end{theorem}

Note that this recovers the Berry-Esseen bound for $d = 1$. Nagaev
\cite{N76} established earlier that this bound is within $d^{1/4}$ of
optimal in the sense that there exist examples which would contradict
the above theorem if $d^{1/4}$ were replaced with some term going to
zero as $d \rightarrow \infty$. However, the family of examples in
\cite{N76} is for a specific relation between $n$, $d$, and $\beta_3$,
which, as we shall see, may not be representative of the behavior of
many natural cases.

Although our result is for $\W_2$ distance, when the dimension
$d$ fixed, convergence in $\W_2$ distance to a Gaussian implies
convergence in probabilities of convex sets.\footnote{On the other
  hand, convergence in probabilities of convex sets does not in
  general imply convergence in $\W_2$ distance, and we do not
  know of any easy way to derive a result similar to Theorem
  \ref{thm:main} from Theorem \ref{thm:bentkus}.} Specifically, we
have the following proposition.

\begin{proposition} \label{prop:w2-to-convex}
  Let $T$ be any $\R^d$-valued random variable, and let $Z$ be a
  standard $d$-dimensional Gaussian. Then, for a universal constant
  $C$,
  \[ \dci(T, Z) \le C d^{\frac{1}{6}} \W_2(T, Z)^{\frac{2}{3}}. \]
\end{proposition}

For the short proof (involving Gaussian surface area of convex sets),
see Appendix \ref{subsec:w2-to-convex-proof}. Applying Proposition
\ref{prop:w2-to-convex} to Theorem \ref{thm:main}, we have the
following corollary.

\begin{corollary} \label{cor:convex}
  Let $X_1, \ldots , X_n$ be independent random vectors in $\R^d$ with
  mean zero, identity covariance, and $\|X_i\| \le \beta$ almost
  surely for each $i$. Let $S_n = \frac{1}{\sqrt{n}} \sum_{i = 1}^n
  X_i$, and let $Z$ be a standard Gaussian. Then, for a universal
  constant $C$,
  \[ \dci(S_n, Z) \le \frac{C d^{\frac{1}{2}} \beta^{\frac{2}{3}} (1 + \log n)^{\frac{2}{3}}}{n^{\frac{1}{3}}}. \]
\end{corollary}

Before we proceed, it should be noted that a few issues arise in
comparing high-dimensional central limit theorems. To start with,
concepts such as ``third moments'' are less clear-cut. For example,
for an $\R^d$-valued random variable $X = (X_1, \ldots , X_d)$, both
$\E \|X\|^3$ and $\sum_{i = 1}^d \E |X_i|^3$ are potentially
reasonable generalizations of the one-dimensional third moment. A
related issue is how to normalize covariances. In the one-dimensional
setting, we can always, without loss of generality, normalize $X$ so
that $\E X^2 = 1$. In higher dimensions, linear transformations on the
covariance matrix have a more complicated effect on quantities such as
the aforementioned third moments.

Corollary \ref{cor:convex} has a suboptimal $n^{-1/3}$ dependence on
$n$ (compared to the correct order $n^{-1/2}$ obtained in Theorem
\ref{thm:bentkus}). Nevertheless, Corollary \ref{cor:convex} yields
better information in some cases. Let us suppose that $\|X_1\| =
\sqrt{d}$ almost surely; this includes natural examples such as when
$X_1$ is $\pm\sqrt{d}$ times a standard basis vector, with the sign
and the basis vector chosen uniformly at random. Then, we have $\beta
= \beta_3 = \sqrt{d}$, so that Theorem \ref{thm:bentkus} gives
\[ \dci(S_n, Z) \le \frac{C d^{7/4}}{n^{1/2}}, \]
while Corollary \ref{cor:convex} gives
\[ \dci(S_n, Z) \le \frac{C d^{5/6} (1 + \log n)^{2/3}}{n^{1/3}}. \]

We find that the second bound is stronger than the first whenever $d = \tilde{\Omega}(n^{2/11})$, where the tilde suppresses logarithmic factors. In particular, note that the
second bound gives $\dci(S_n, Z) = o(1)$ (i.e. says something non-trivial) as soon as $d = \tilde{o}(n^{2/5})$,
while the first bound requires $d = o(n^{2/7})$.\footnote{We remark that even if the $d^{1/4}$ in Theorem
  \ref{thm:bentkus} were replaced by a constant as in Nagaev's lower
  bound, it would only give $\dci(S_n, Z) = o(1)$ for $d = o(n^{1/3})$, which is still more restrictive than $d = o(n^{2/5})$. Thus, Corollary \ref{cor:convex} proves that under the assumption $\|X_1\| = \sqrt{d}$, convergence in $\dci$ is actually faster than indicated by Nagaev's example (which does not satisfy $\|X_1\| = \sqrt{d}$).}
In this sense, when $\|X_1\| = \sqrt{d}$ almost
surely, Corollary \ref{cor:convex} gives convergence for a larger range of $d$.

We mention here that in high-dimensional settings, $d$ may indeed be as large as a power of $n$. For example, the
earlier mentioned work of Bubeck and Ganguly \cite{BG15}, when applied in the context of \cite{BDER16}, concerns $d
\approx n^{2/3}$ (after converting to our notation). The work of
Chernozhukov, Chetverikov, and Kato \cite{CCK13} even considers $d
\approx e^{n^c}$ for a constant $c$, albeit working under a much
weaker notion of convergence.

\subsection{Idea of the proof} \label{subsec:proof-idea}

The proof of Theorem \ref{thm:main} follows a Lindeberg-type strategy
of gradually replacing $X_i$'s with Gaussians. However, instead of
working with sufficiently smooth test functions, we directly compare
probability densities. A major ingredient for accomplishing this is
Talagrand's transportation inequality. To our knowledge, this
variation of the Lindeberg strategy has not appeared before in the
literature, and the idea may be of use in other settings. Our argument
rests upon the following key lemma, which bounds the error arising
from replacing $X_i$ with a Gaussian.

\begin{lemma} \label{lem:increment}
  Let $X$ be a $\R^k$-valued random variable with mean $0$, covariance
  $\Sigma$, and $\|X\| \le \beta$ almost surely. Let $Z_t$ denote a
  Gaussian of mean $0$ and covariance $t \Sigma$ independent of
  $X$. Let $\sigmamin^2$ denote the smallest eigenvalue of
  $\Sigma$. Then, for any $n \ge \frac{5 \beta^2}{\sigmamin^2}$, we
  have
  \[ \W_2(Z_n, Z_{n - 1} + X) \le \frac{5 \sqrt{k} \beta}{n}. \]
\end{lemma}

\begin{remark} \label{rem:n-k-bound}
  The assumption on $n$ implies that $n \ge 5k$, because
  \[ n \ge \frac{5 \beta^2}{\sigma_k^2} \ge \frac{5}{\sigma_k^2} \E \|X\|^2 = \frac{5}{\sigma_k^2} \sum_{i = 1}^k \sigma_i^2 \ge 5k. \]
\end{remark}

Heuristically, Lemma \ref{lem:increment} says that when you add an
independent random variable $X$ to a Gaussian $Z_{n - 1}$, the
resulting distribution is still nearly Gaussian. The hypothesis that
$n$ be sufficiently large is required to ensure that $X$ is small
compared to $Z_{n - 1}$. Note that the dimension $k$ appearing in
Lemma \ref{lem:increment} is not necessarily equal to $d$. This is a
subtle but important point---we will selectively apply the estimate of
Lemma \ref{lem:increment} to only a subset of the coordinates
depending on the variance of $X$ in those directions.

Theorem \ref{thm:main} follows from repeated applications of Lemma
\ref{lem:increment}. To prove Lemma \ref{lem:increment}, our strategy
is to take advantage of the fact that we can explicitly compute the
density of the Gaussian $Z_n$, and we also have a fairly explicit form
for the density of $Z_{n - 1} + X$. We can then make precise density
estimates, which are conveniently translated into $\W_2$ estimates via
(a variant of) Talagrand's transportation inequality.

\subsection{Organization of the paper}

The rest of the paper is organized as follows. In Section
\ref{sec:main-proof}, we prove Theorem \ref{thm:main} assuming Lemma
\ref{lem:increment}. In Section \ref{sec:transportation}, we provide
some background on Talagrand's transportation inequality needed to
prove Lemma \ref{lem:increment}. In particular, whereas the
inequality is usually formulated in the setting of a standard
$n$-dimensional Gaussian, we give a version for general
Gaussians. Finally, Section \ref{sec:lemma-proof} gives the proof of
Lemma \ref{lem:increment}, filling in the technical details of the
strategy described above.

\subsection{Acknowledgements}

We are indebted to Jian Ding for suggesting the use of Talagrand's
transportation inequality and Amir Dembo for pointing out a hole in a
preliminary version of the main argument as well as many helpful
comments on the exposition. We also thank Sourav Chatterjee for
helpful discussions about related work. Finally, we thank the
anonymous reviewers for many good suggestions and for pointing out
several references.

\section{Proof of Theorem \ref{thm:main}} \label{sec:main-proof}

We first show how to deduce Theorem \ref{thm:main} from Lemma
\ref{lem:increment}. Recall however that the statement of Lemma
\ref{lem:increment} contains a hypothesis that $n \ge \frac{5
  \beta^2}{\sigmamin^2}$. Thus, we will also need an \textit{a priori}
bound to estimate $\W_2$ distances for smaller $n$.

Luckily, a na\"{\i}ve bound suffices. For any mean-zero random
variables $X$ and $Y$, coupling them to be independent yields the
inequality $\W_2(X, Y)^2 \le \E \|X\|^2 + \E \|Y\|^2$. The next lemma
is a slight refinement of this observation to consider only a subset
of coordinates.

\begin{lemma} \label{lem:naive-W_2-bound}
  Let $X = (X_1, \ldots , X_d)$ and $Y = (Y_1, \ldots , Y_d)$ be two
  $\R^d$-valued random variables with mean zero. Moreover, suppose
  that $\E\( (Y_{k+1}, \ldots , Y_d) \mid Y_1, \ldots , Y_k \) = 0$. Then,
  \[ \W_2(X, Y)^2 \le \W_2((X_1, \ldots , X_k), (Y_1, \ldots , Y_k))^2 + \sum_{i = k + 1}^d \( \E X_i^2 + \E Y_i^2 \). \]
\end{lemma}
\begin{proof}
  For convenience, define $P_k : \R^d \to \R^d$ by $P_k(x_1, \ldots ,
  x_d) = (x_1, \ldots , x_k, 0, \ldots , 0)$. Let $\tilde{X}$ and
  $\tilde{Y}$ be a coupling of $X$ and $Y$ given by first sampling
  $P_k(\tilde{X})$ and $P_k(\tilde{Y})$ according to a coupling such
  that
  \[ \E \| P_k(\tilde{X}) - P_k(\tilde{Y}) \|^2 = \W_2(P_k(X), P_k(Y)) \]
  and then sampling $\tilde{X}$ and $\tilde{Y}$ independently
  conditioned on $P_k(\tilde{X})$ and $P_k(\tilde{Y})$. Thus,
  $\tilde{X} - P_k(\tilde{X})$ and $\tilde{Y} - P_k(\tilde{Y})$ are
  independent conditioned on $P_k(\tilde{X})$ and
  $P_k(\tilde{Y})$. Then,
  \[ \W_2(X, Y)^2 \le \E \| \tilde{X} - \tilde{Y} \|^2 = \E \| \tilde{X} - \tilde{Y} \|^2 \]
  \[ = \E \| (\tilde{X} - P_k(\tilde{X})) + (P_k(\tilde{X}) - P_k(\tilde{Y})) + (P_k(\tilde{Y}) - \tilde{Y}) \|^2 \]
  \[ = \E \| \tilde{X} - P_k(\tilde{X}) \|^2 + \E \| P_k(\tilde{X}) - P_k(\tilde{Y}) \|^2 + \E \| P_k(\tilde{Y}) - \tilde{Y} \|^2 \]
  \[ = \W_2(P_k(X), P_k(Y))^2 + \sum_{i = k + 1}^d \( \E X_i^2 + \E Y_i^2 \). \]
\end{proof}

We are now ready for the main proof. The rough idea is to induct
simultaneously on $n$ and the dimension. At each step, if possible, we
apply Lemma \ref{lem:increment} to increase $n$. Otherwise, we apply
Lemma \ref{lem:naive-W_2-bound} to increase the dimension.

\begin{proof}[Proof of Theorem \ref{thm:main}]
  Using the notation in the statement of the theorem, we can assume
  without loss of generality that $\Sigma$ takes the form
  \[ \Sigma = \begin{bmatrix}
    \sigma_1^2 & 0 & \cdots & 0 \\
    0 & \sigma_2^2 & \cdots & 0 \\
    \vdots & \vdots & \ddots & \vdots \\
    0 & 0 & \cdots & \sigma_d^2 \\
  \end{bmatrix}, \]
  with $\sigma_1 \ge \sigma_2 \ge \cdots \ge \sigma_d > 0$. For each
  $n \ge 1$, define
  \[ S_n = \sum_{i = 1}^n X_i, \]
  and let $Z_n$ denote a Gaussian with covariance $n \Sigma$.

  Let $P_k : \R^d \to \R^k$ denote the projection onto the first $k$
  coordinates, and for $0 \le k \le d$, define
  \[ A_{n,k} = \W_2(P_k(S_n), P_k(Z_n)), \qquad A_{n,0} = 0. \]
  We will prove by induction on $n$ and $k$ that
  \begin{equation} \label{eq:nk-bound}
    A_{n,k} \le 5 \sqrt{k} \beta (1 + \log n)
  \end{equation}
  for all $n \ge 1$ and $0 \le k \le d$. The theorem then follows by
  taking $k = d$.

  Let us call $(n, k)$ a \emph{good pair} if \eqref{eq:nk-bound}
  holds. We first prove the base cases. If $k = 0$, then
  \eqref{eq:nk-bound} holds trivially. If $n = 1$, then by Lemma
  \ref{lem:naive-W_2-bound},
  \[ A_{1,k} = \W_2(P_k(X_1), P_k(Z_1)) \le \sqrt{\E \|X_1\|^2 + \E \|Z_1\|^2} \le 2 \beta, \]
  so again \eqref{eq:nk-bound} holds.

  For the inductive step, consider any $n > 1$ and $k > 0$. Our
  inductive hypothesis is that $(n - 1, k)$ and $(n, k - 1)$ are good
  pairs, and we will show that $(n, k)$ is a good pair as well. If $n
  > \frac{5 \beta^2}{\sigma_k^2}$, then we may apply Lemma
  \ref{lem:increment} to $P_k(X_n)$, whose covariance is just the
  top-left $k \times k$ submatrix of $\Sigma$. This gives
  \[ \W_2(P_k(Z_{n - 1} + X_n), P_k(Z_n)) \le \frac{5\sqrt{k} \beta}{n}. \]
  Consequently,
  \begin{align*}
    A_{n,k} &= \W_2(P_k(S_n), P_k(Z_n)) = \W_2(P_k(S_{n - 1} + X_n), P_k(Z_n)) \\
    &\le \W_2(P_k(S_{n - 1} + X_n), P_k(Z_{n - 1} + X_n)) + \W_2(P_k(Z_{n - 1} + X_n), P_k(Z_n)) \\
    &\le A_{n-1,k} + \frac{5\sqrt{k} \beta}{n} \le 5 \sqrt{k} \beta \( 1 + \log (n - 1) + \frac{1}{n} \) \le 5 \sqrt{k} \beta (1 + \log n). \\
  \end{align*}
  Otherwise, if $n \le \frac{5 \beta^2}{\sigma_k^2}$, then by Lemma
  \ref{lem:naive-W_2-bound}, we have
  \[ A_{n,k}^2 \le A_{n,k-1}^2 + 2n \sigma_k^2 \]
  \[ \le 25 (k - 1) \beta^2 (1 + \log n)^2 + 10 \beta^2 \le 25 k \beta^2 (1 + \log n)^2. \]
  We see in both cases that $(n, k)$ is a good pair, completing the
  induction and the proof.
\end{proof}

\section{A transportation inequality} \label{sec:transportation}

It remains only to prove Lemma \ref{lem:increment}. As described
earlier, the strategy we use is to translate closeness in probability
densities into closeness in $\W_2$ distance. In this section, we
establish the result needed for this purpose, which is based on the
following inequality due to Talagrand.

\begin{theorem}[Talagrand's transportation inequality] \label{thm:talagrand-orig}
  Let $Z$ be a standard $d$-dimensional Gaussian with density
  $\rho$. Let $\mu$ be a probability density on $\R^d$ and let $f(x) =
  \frac{d \mu}{d \rho}(x)$. Then,
  \[ \W_2(\mu, \rho)^2 \le 2 \E f(Z) \log f(Z) = 2 \KL{\mu}{\rho}. \]
\end{theorem}

\begin{remark}
  We note that the above inequality is sharp: equality holds when $Y$
  is Gaussian with the same covariance as $Z$, but with a different
  mean. However, it can be far from optimal when the density of $Y$ is
  not very ``smooth''; indeed, in the extreme case where $Y$ is not
  absolutely continuous with respect to $Z$, Theorem
  \ref{thm:talagrand-orig} says nothing at all. The need to ensure
  this ``smoothness'' explains the requirement that $n \ge \frac{5
    \beta^2}{\sigmamin^2}$ in the statement of Lemma
  \ref{lem:increment}.
\end{remark}

Theorem \ref{thm:talagrand-orig} is an example of a
transportation-information inequality (also known as
transportation-cost inequalities in the literature). Such inequalities
were first studied by Marton \cite{M96} who showed their connection to
concentration of measure phenomena (see also \cite{BG99}).

In \cite{T96}, Talagrand proves Theorem \ref{thm:talagrand-orig} using
an inductive argument, following ideas of Marton \cite{M96}. The
one-dimensional case is a (non-trivial!) calculus problem. Higher
dimensions then follow by tensorization properties of $\W_2$ distance
and relative entropy.

However, we cannot directly apply Talagrand's transportation
inequality in our case, because the covariance of our Gaussian is not
the identity. Nevertheless, by modifying the proof only slightly, we
can obtain a version of the inequality that applies to non-standard
Gaussians, as captured in the next proposition.

\begin{proposition}[variant of Talagrand's transportation inequality] \label{prop:talagrand}
  Let $Z$ be a $d$-dimensional Gaussian having diagonal covariance
  \[ \Sigma = \begin{bmatrix}
    \sigma_1^2 & 0 & \cdots & 0 \\
    0 & \sigma_2^2 & \cdots & 0 \\
    \vdots & \vdots & \ddots & \vdots \\
    0 & 0 & \cdots & \sigma_d^2 \\
  \end{bmatrix} \]
  with $\sigma_1 \ge \sigma_2 \ge \cdots \ge \sigma_d > 0$. Let $\rho:
  \R^d \to \R$ be the density of $Z$, and let $Y$ be a $\R^d$-valued
  random variable with density $f(x) \rho(x)$. Then,
  \[ \W_2(Y, Z)^2 \le 2 \sum_{i = 1}^d \sigma_i^2 \( \E f(Z)^2 - \E f_{(i)}(Z)^2 \), \]
  where $f_{(i)}$ is the ``averaging'' of $f$ along the $i$-th
  coordinate defined by
  \[ f_{(i)}(x) = \frac{\displaystyle
    \int_{-\infty}^\infty f(x + te_i) \rho(x + te_i) \,dt
  }{\displaystyle
    \int_{-\infty}^\infty \rho(x + te_i) \,dt
  }, \]
  where $e_i \in \R^d$ are unit coordinate vectors.
\end{proposition}

The proof of Proposition \ref{prop:talagrand} uses an elementary lemma
involving conditional $L^2$ norms, which is proved in Appendix
\ref{subsec:conditional-L^2-lemma}.

\begin{lemma} \label{lem:conditional-L^2-lemma}
  Let $A \in \mathcal{A}$ and $B \in \mathcal{B}$ be independent
  random variables and consider any function $f : \mathcal{A} \times \mathcal{B} \to \R$. Define
  \[ f_A : \mathcal{B} \to \R, \qquad f_A(b) = \E\( f(A, B) \mid B = b \) \]
  \[ f_B : \mathcal{A} \to \R, \qquad f_B(a) = \E\( f(A, B) \mid A = a \). \]
  Then,
  \[ \E f(A, B)^2 + (\E f(A, B))^2 \ge \E f_A(B)^2 + \E f_B(A)^2. \]
\end{lemma}

\begin{proof}[Proof of Proposition \ref{prop:talagrand}]
  In fact, a slightly stronger inequality holds. In order to state it,
  let us define for each $0 \le k \le d$ the function
  \[ f_{[k]}(x) = \frac{\displaystyle
    \int_{-\infty}^\infty \cdots \int_{-\infty}^\infty f\(x + \sum_{i = k + 1}^d t_ie_i\) \rho\(x + \sum_{i = k + 1}^d t_ie_i\) \,dt_{k + 1} \cdots \,dt_d
    }{\displaystyle
    \int_{-\infty}^\infty \cdots \int_{-\infty}^\infty \rho\(x + \sum_{i = k + 1}^d t_ie_i\) \,dt_{k + 1} \cdots \,dt_d
  }, \]
  which may be thought of as the ``averaging'' of $f$ over all but the
  first $k$ coordinates. Note that $f_{[d]} = f$ and $f_{[0]} = 1$.

  We claim that
  \begin{equation} \label{eq:entropy-bound}
    \W_2(Y, Z)^2 \le 2 \sum_{k = 1}^d \sigma_k^2 \cdot \E \( f_{[k]}(Z) \log f_{[k]}(Z) - f_{[k - 1]}(Z) \log f_{[k - 1]}(Z) \).
  \end{equation}
  This inequality is essentially a byproduct of the proof of Theorem
  \ref{thm:talagrand-orig} (see \cite{T96}, \S 3). Note that if
  $\sigma_k = 1$ for all $k$, then the sum in \eqref{eq:entropy-bound}
  telescopes to
  \[ 2 \E f(Z) \log f(Z) = 2 \KL{Y}{Z}, \]
  recovering Theorem \ref{thm:talagrand-orig}. Although
  \eqref{eq:entropy-bound} is a direct consequence of the arguments in
  \cite{T96}, for the sake of completeness we repeat the proof in
  Appendix \ref{subsec:entropy-bound-proof}.

  Using \eqref{eq:entropy-bound} and the fact that $t \log t \le t^2 -
  t$, we have
  \begin{align*}
    \W_2(Y, Z)^2 &\le 2 \sum_{k = 1}^d \sigma_k^2 \cdot \E \( f_{[k]}(Z) \log f_{[k]}(Z) - f_{[k - 1]}(Z) \log f_{[k - 1]}(Z) \) \\
    &= 2 \sigma_d^2 \cdot \E \( f_{[d]}(Z) \log f_{[d]}(Z) \) + 2 \sum_{k = 2}^d (\sigma_{k - 1}^2 - \sigma_k^2) \E \( f_{[k - 1]}(Z) \log f_{[k - 1]}(Z) \) \\
    &\le 2 \sigma_d^2 \cdot \E \( f_{[d]}(Z)^2 - f_{[d]}(Z)\) + 2 \sum_{k = 2}^d (\sigma_{k - 1}^2 - \sigma_k^2) \E \(f_{[k - 1]}(Z)^2 - f_{[k - 1]}(Z)\) \\
    &= 2 \sigma_d^2 \cdot \E \( f_{[d]}(Z)^2 - 1\) + 2 \sum_{k = 2}^d (\sigma_{k - 1}^2 - \sigma_k^2) \E \(f_{[k - 1]}(Z)^2 - 1\) \\
    &= 2 \sum_{k = 1}^d \sigma_k^2 \cdot \E \( f_{[k]}(Z)^2 - f_{[k - 1]}(Z)^2 \)
  \end{align*}
  Finally, for each $k$, we claim that
  \begin{equation} \label{eq:conditional-L^2-bound}
    \E \( f_{[k]}(Z)^2 - f_{[k - 1]}(Z)^2 \) \le \E \( f(Z)^2 - f_{(k)}(Z)^2 \).
  \end{equation}
  Indeed, this is actually an immediate consequence of Lemma
  \ref{lem:conditional-L^2-lemma}. To simplify notation, write $Z =
  (Z', Z'', Z''')$, where $Z'$ denotes the first $k - 1$ coordinates
  of $Z$, $Z''$ denotes the $k$-th coordinate, and $Z'''$ denotes the
  last $d - k$ coordinates. In terms of these variables, we have
  \begin{align*}
    \E f_{[k - 1]}(Z)^2 &= \E \left[ \E\( f(Z) \mid Z' \)^2 \right] \\
    \E f_{[k]}(Z)^2 &= \E \left[ \E\( f(Z) \mid Z', Z'' \)^2 \right] \\
    \E f_{(k)}(Z)^2 &= \E \left[ \E\( f(Z) \mid Z', Z''' \)^2 \right] \\
    \E f(Z)^2 &= \E \left[ \E\( f(Z)^2 \mid Z' \) \right],
  \end{align*}
  Then, applying Lemma \ref{lem:conditional-L^2-lemma} conditioned on
  $Z'$ with $A = Z''$ and $B = Z'''$ gives us precisely
  \eqref{eq:conditional-L^2-bound}. Thus, we conclude that
  \[ \W_2(Y, Z)^2 \le 2 \sum_{k = 1}^d \sigma_k^2 \cdot \E \( f(Z)^2 - f_{(k)}(Z)^2 \), \]
  as desired.
\end{proof}

\section{Proof of Lemma \ref{lem:increment}} \label{sec:lemma-proof}

We finally conclude by proving Lemma \ref{lem:increment}. Henceforth,
we use the notation in the statement of Lemma \ref{lem:increment} and
assume without loss of generality that
\[ \Sigma = \begin{bmatrix}
  \sigma_1^2 & 0 & \cdots & 0 \\
  0 & \sigma_2^2 & \cdots & 0 \\
  \vdots & \vdots & \ddots & \vdots \\
  0 & 0 & \cdots & \sigma_k^2 \\
\end{bmatrix}, \]
so that $\sigmamin = \min_{1 \le i \le k} \sigma_i$. It is more
convenient to work with the normalization $Y = \frac{1}{\sqrt{n}} X$,
so that $\|Y\| \le \frac{\beta}{\sqrt{n}}$. Our goal is then to prove
that
\[ \W_2(Z_1, Z_{1 - 1/n} + Y) \le \frac{5 \sqrt{k} \beta}{n \sqrt{n}} \]
for $n \ge \frac{5 \beta^2}{\sigmamin^2}$.

\subsection{A density computation}

The goal of this subsection is to explicitly compute the density of
$Z_{1 - 1/n} + Y$ and its marginals needed to apply Proposition
\ref{prop:talagrand}. We will want to use the approximation
\[ \frac{1}{2} \log \( 1 + \frac{1}{n^2 - 1} \) \approx \frac{1}{2(n^2 - 1)}. \]
To this end, it is convenient to define
\[ r(n) = \frac{1}{2(n^2 - 1)} - \frac{1}{2} \log \( 1 + \frac{1}{n^2 - 1} \). \]
Note that since $t - 2t^2 \le \log(1 + t) \le t$ for any $t \ge 0$, we
have for any $n \ge 2$ that
\[ 0 \le r(n) \le \frac{1}{(n^2 - 1)^2}. \]
The following lemma gives the formula for the density of $Z_{1 - 1/n}
+ Y$.

\begin{lemma} \label{lem:L^2-formula}
  Let $\rho$ be the density of $Z_1$, let $\tau$ be the density of
  $Z_{1 - 1/n} + Y$, and let $f(x) = \frac{\tau(x)}{\rho(x)}$. Then,
  \[ \E f(Z)^2 = \E \left[ \exp\( \sum_{i = 1}^k \frac{2n^2 Y_iY'_i - n Y_i^2 - n (Y'_i)^2 + \sigma_i^2}{2 \sigma_i^2 (n^2 - 1)} - r(n) \) \right], \]
  where $Y'$ is an independent copy of $Y$.
\end{lemma}

The proof is a straightforward calculation based on the following
computational lemma, proved in Appendix
\ref{subsec:gaussian-expectation-proof}.

\begin{lemma} \label{lem:gaussian-expectation}
  Let $Z$ be a $k$-dimensional Gaussian with covariance
  $\Sigma$. Define $\langle u, v \rangle_{\Siginv} = \langle u,
  \Siginv v \rangle$ and $\| u \|_{\Siginv} = \sqrt{\langle u, u
    \rangle_{\Siginv}}$. Then,
  \[ \E \left[ \exp\( a\|Z\|_\Siginv^2 + b\langle Z, v \rangle_\Siginv \) \right] = \exp\( \frac{b^2}{2 - 4a} \|v\|_\Siginv^2 \) \cdot \( \frac{1}{1 - 2a} \)^{k/2}. \]
\end{lemma}

\begin{proof}[Proof of Lemma \ref{lem:L^2-formula}]
  In the notation of Lemma \ref{lem:gaussian-expectation}, the formula
  for $\rho$ is
  \[ \rho(x) = \frac{1}{\sqrt{(2 \pi)^k \cdot \det \Sigma}} \exp\( - \frac{1}{2} \|x\|^2_\Siginv \). \]
  We write can $\tau$ in terms of $\rho$ by
  \begin{align*}
    \tau(x) &= \E \left[ \frac{1}{\(1 - 1/n\)^{k/2}} \cdot \rho \( \frac{x - Y}{\sqrt{1 - 1/n}}  \) \right] \\
    &= \E \left[ \frac{1}{\(1 - 1/n\)^{k/2}} \exp\( - \frac{1}{2 - 2/n}\|x - Y\|_\Siginv^2 + \frac{1}{2}\|x\|_\Siginv^2 \) \rho(x) \right] \\
    &= \E \left[ \frac{1}{\(1 - 1/n\)^{k/2}} \exp\( - \frac{\|x\|_\Siginv^2}{2n - 2} + \frac{n \langle x, Y \rangle_\Siginv}{n - 1} - \frac{n\|Y\|_\Siginv^2}{2n - 2} \) \right] \rho(x) \\
  \end{align*}
  Then, we have
  \[ f(x) = \frac{\tau(x)}{\rho(x)} = \E \left[ \frac{1}{\(1 - 1/n\)^{k/2}} \exp\( - \frac{\|x\|_\Siginv^2}{2n - 2} + \frac{n \langle x, Y \rangle_\Siginv}{n - 1} - \frac{n\|Y\|_\Siginv^2}{2n - 2} \) \right]. \]
  It follows that
  \begin{align*}
    \E f(Z)^2 &= \(1 - 1/n\)^{-k} \cdot \E \left[ \exp\( - \frac{\|Z\|_\Siginv^2}{2n - 2} + \frac{n \langle x, Y \rangle_\Siginv}{n - 1} - \frac{n\|Y\|_\Siginv^2}{2n - 2} \) \right]^2 \\
    &= \(1 - 1/n\)^{-k} \cdot \E \left[ \exp\( - \frac{\|Z\|_\Siginv^2}{n - 1} + \frac{n \langle Z, Y + Y' \rangle_\Siginv}{n - 1} \) \right] \\
    & \hphantom{= \(1 - 1/n\)^{-k}\;} \cdot\, \E \left[ \exp\( - \frac{n(\|Y\|_\Siginv^2 + \|Y'\|_\Siginv^2)}{2n - 2} \) \right],
  \end{align*}
  where we have used the fact that for any function $\alpha$, $\( \E\,
  \alpha(Y) \)^2 = \E\(\alpha(Y) \alpha(Y')\)$.

  We apply Lemma \ref{lem:gaussian-expectation} with $a = -\frac{1}{n
    - 1}$, $b = \frac{n}{n - 1}$, and $v = Y + Y'$. Note that $1 - 2a
  = 1 + \frac{2}{n - 1} = \frac{n + 1}{n - 1}$. The above expression
  then becomes
  \begin{align*}
    \E f(Z)^2 &= \(1 - 1/n\)^{-k} \cdot \( \frac{n - 1}{n + 1} \)^{k/2} \E \left[ \exp\( \frac{n^2 \|Y + Y'\|_\Siginv^2}{2(n^2 - 1)}\) \right] \\
    &\hphantom{= \(1 - 1/n\)^{-k}\;} \cdot \E\left[ \exp\(-\frac{n(\|Y\|_\Siginv^2 + \|Y'\|_\Siginv^2)}{2(n - 1)} \) \right] \\
    &= \( \frac{n^2}{n^2 - 1} \)^{k/2} \E \left[ \exp\( \frac{2n^2 \langle Y, Y' \rangle_\Siginv - n \|Y\|_\Siginv^2 - n \|Y'\|_\Siginv^2}{2(n^2 - 1)}\) \right] \\
    &= \E \left[ \exp\( \frac{2n^2 \langle Y, Y' \rangle_\Siginv - n \|Y\|_\Siginv^2 - n \|Y'\|_\Siginv^2}{2(n^2 - 1)} + \frac{k}{2} \log \( 1 + \frac{1}{n^2 - 1} \) \) \right] \\
    &= \E \left[ \exp\( \sum_{i = 1}^k \frac{2n^2 Y_iY'_i - n Y_i^2 - n (Y'_i)^2}{2 \sigma_i^2 (n^2 - 1)} + \frac{1}{2} \log \( 1 + \frac{1}{n^2 - 1} \) \) \right] \\
    &= \E \left[ \exp\( \sum_{i = 1}^k \frac{2n^2 Y_iY'_i - n Y_i^2 - n (Y'_i)^2 + \sigma_i^2}{2 \sigma_i^2 (n^2 - 1)} - r(n) \) \right]
  \end{align*}
\end{proof}

Note that any projection of $Z_{1 - 1/n} + Y$ onto a subset of its
coordinates still takes the form of a Gaussian plus an independent
random vector. Therefore, Lemma \ref{lem:L^2-formula} can also be
applied to projections of $Z$ and $Y$, leading to the following
corollary.

\begin{corollary} \label{cor:L^2-formula}
  Let $Y_i$ denote the $i$-th coordinate of $Y$. For each $1 \le i \le
  k$, define
  \[ Q_i = \frac{2n^2 Y_iY'_i - n Y_i^2 - n (Y'_i)^2 + \sigma_i^2}{2 \sigma_i^2 (n^2 - 1)} - r(n), \qquad Q = \sum_{i = 1}^k Q_i. \]
  Then, for each $i$, we have
  \[ \E f_{(i)}(Z)^2 = \E \exp\( \sum_{j \ne i} Q_j \) = \E \exp\( Q - Q_i \), \]
  where the notation $f_{(i)}$ follows that of Proposition
  \ref{prop:talagrand}.
\end{corollary}
\begin{proof}
  Let $P_{(i)} : \R^k \to \R^{k - 1}$ denote the projection onto all
  but the $i$-th coordinate. Then, the result follows by replacing
  $Z_{1 - 1/n}$ and $Y$ in Lemma \ref{lem:L^2-formula} with
  $P_{(i)}(Z_{1 - 1/n})$ and $P_{(i)}(Y)$, respectively.
\end{proof}

\subsection{Some computational estimates of the $Q_i$}

Our strategy was to bound $\W_2$ distance via Proposition
\ref{prop:talagrand}, which reduces the problem to estimating various
densities. By Lemma \ref{lem:L^2-formula} and Corollary
\ref{cor:L^2-formula}, we have now expressed the densities of interest
in terms of the quantities $Q_i$, so the next step is to estimate the
$Q_i$. In what follows, recall that we assumed $n \ge \frac{5
  \beta^2}{\sigma_i^2}$ for each $i$, and consequently, $n \ge 5k$
(see Remark \ref{rem:n-k-bound}). Also, recall that by assumption we
have
\[ \|Y\| \le \frac{\beta}{\sqrt{n}}, \quad \E Y_i = 0, \quad\text{and}\quad \E Y_i^2 = \frac{\sigma_i^2}{n}. \]
The bounds we obtain are summarized in the next two lemmas.

\begin{lemma} \label{lem:|Q|-estimates}
  We have
  \[ |Q_i| \le \frac{n^2 |Y_iY'_i|}{\sigma_i^2(n^2 - 1)} + \frac{1}{2n}, \qquad |Q| \le 1, \quad\text{and}\quad |Q - Q_i| \le 1. \]
\end{lemma}
\begin{proof}
  To prove the first inequality, we have
  \begin{align*}
    |Q_i| &= \left| \frac{2n^2 Y_iY'_i - n Y_i^2 - n (Y'_i)^2 + \sigma_i^2}{2 \sigma_i^2 (n^2 - 1)} - r(n) \right| \\
    &\le \frac{n^2 |Y_iY'_i|}{\sigma_i^2(n^2 - 1)} + \frac{\beta^2}{\sigma_i^2(n^2 - 1)} + \frac{1}{n^2 - 1} + r(n) \\
    &\le \frac{n^2 |Y_iY'_i|}{\sigma_i^2(n^2 - 1)} + \frac{n}{5(n^2 - 1)} + \frac{1}{n^2 - 1} + \frac{1}{(n^2 - 1)^2} \\
    &\le \frac{n^2 |Y_iY'_i|}{\sigma_i^2(n^2 - 1)} + \frac{1}{2n}.
  \end{align*}
  Summing over all $i$, we obtain
  \begin{align*}
    |Q| &\le \sum_{i = 1}^k |Q_i| \le \sum_{i = 1}^k \( \frac{n^2 |Y_iY'_i|}{\sigma_i^2(n^2 - 1)} + \frac{1}{2n} \) \le \frac{n^2}{\sigmamin^2 (n^2 - 1)} \( \sum_{i = 1}^k |Y_iY'_i| \) + \frac{k}{2n} \\
    &\le \frac{n \beta^2}{\sigmamin^2(n^2 - 1)} + \frac{k}{2n} \le \frac{n^2}{5(n^2 - 1)} + \frac{1}{2} \le 1,
  \end{align*}
  proving the second inequality. The third inequality follows by a
  similar argument, except that we omit one of the $|Q_i|$ terms in
  the sum.
\end{proof}

\begin{lemma} \label{lem:Q-estimates}
  We have
  \begin{align}
    \E Q_i &= - \frac{1}{2(n^2 - 1)} - r(n) \label{eq:Q_i-bound} \\
    \E Q_i Q_j &\le \frac{n^2}{(n^2 - 1)^2} \delta_{ij} + \frac{n^2 \E Y_i^2 Y_j^2}{2 \sigma_i^2 \sigma_j^2 (n^2 - 1)^2} + \frac{1}{2(n^2 - 1)^2} \label{eq:Q_iQ_j-bound} \\
    \E Q_i^2 &\le \frac{2n^2 + n + 1}{2(n^2 - 1)^2} \label{eq:Q_i^2-bound} \\
    \E (Q - Q_i)Q_i &\le \frac{nk}{2(n^2 - 1)^2} \label{eq:(Q - Q_i)Q_i-bound} \\
    \E Q^2 &\le \frac{2k}{n^2 - 1}. \label{eq:Q^2-bound}
  \end{align}
\end{lemma}

\begin{proof}
  To show \eqref{eq:Q_i-bound}, we may compute
  \[ \E Q_i = \frac{-\sigma_i^2 - \sigma_i^2 + \sigma_i^2}{2 \sigma_i^2(n^2 - 1)} - r(n) = - \frac{1}{2(n^2 - 1)} - r(n). \]
  To show \eqref{eq:Q_iQ_j-bound}, we have
  \begin{align*}
    \E Q_i Q_j &\le \E \left[ \(Q_i - \frac{1}{2(n^2 - 1)} + r(n) \) \(Q_j - \frac{1}{2(n^2 - 1)} + r(n) \) \right] \\
    &\le \E \left[ \( \frac{2n^2 Y_iY'_i - nY_i^2 - n(Y'_i)^2}{2 \sigma_i^2 (n^2 - 1)} \) \( \frac{2n^2 Y_jY'_j - nY_j^2 - n(Y'_j)^2}{2 \sigma_j^2 (n^2 - 1)} \) \right] \\
    &= \frac{4n^4 \E Y_iY_jY'_iY'_j + n^2 \E \left[ (Y_i^2 + (Y'_i)^2)(Y_j^2 + (Y'_j)^2) \right]}{4 \sigma_i^2 \sigma_j^2 (n^2 - 1)^2} \\
    &= \frac{4n^2 \sigma_i^2 \sigma_j^2 \delta_{ij} + 2n^2 \E Y_i^2 Y_j^2 + 2\sigma_i^2 \sigma_j^2}{4 \sigma_i^2 \sigma_j^2 (n^2 - 1)^2} \\
    &= \frac{n^2}{(n^2 - 1)^2} \delta_{ij} + \frac{n^2 \E Y_i^2 Y_j^2}{2 \sigma_i^2 \sigma_j^2 (n^2 - 1)^2} + \frac{1}{2(n^2 - 1)^2}. \\
  \end{align*}
  Finally, we can deduce \eqref{eq:Q_i^2-bound}, \eqref{eq:(Q -
    Q_i)Q_i-bound}, and \eqref{eq:Q^2-bound} from
  \eqref{eq:Q_iQ_j-bound}. Setting $i = j$ in \eqref{eq:Q_iQ_j-bound}
  yields
  \begin{align*}
    \E Q_i^2 &\le \frac{n^2}{(n^2 - 1)^2} + \frac{n^2 \E Y_i^4}{2 \sigma_i^4 (n^2 - 1)^2} + \frac{1}{2(n^2 - 1)^2} \\
    &\le \frac{2n^2 + 1}{2(n^2 - 1)^2} + \frac{n \beta^2 \E Y_i^2}{2 \sigma_i^4 (n^2 - 1)^2} \\
    &= \frac{2n^2 + 1 + \beta^2 / \sigma_i^2}{2(n^2 - 1)^2} \le \frac{2n^2 + n + 1}{2(n^2 - 1)^2}, \\
  \end{align*}
  proving \eqref{eq:Q_i^2-bound}. If instead we sum
  \eqref{eq:Q_iQ_j-bound} over all $j \ne i$, we obtain
  \begin{align*}
    \E (Q - Q_i)Q_i &\le \frac{n^2}{2 \sigma_i^2 (n^2 - 1)^2} \E \( Y_i^2 \sum_{j \ne i} \frac{1}{\sigma_j^2} Y_j^2 \) + \frac{k - 1}{2(n^2 - 1)^2} \\
    &\le \frac{n \beta^2}{2 \sigma_i^2 (n^2 - 1)^2} \E \( \sum_{j \ne i} \frac{1}{\sigma_j^2} Y_j^2 \) + \frac{k - 1}{2(n^2 - 1)^2} \\
    &= \frac{(k - 1)(\beta^2 / \sigma_i^2 + 1)}{2(n^2 - 1)^2} \le \frac{nk}{2(n^2 - 1)^2}, \\
  \end{align*}
  proving \eqref{eq:(Q - Q_i)Q_i-bound}. Finally, adding
  \eqref{eq:Q_i^2-bound} and \eqref{eq:(Q - Q_i)Q_i-bound} and summing
  over all $i$, we obtain
  \begin{align*}
    \E Q^2 &= \sum_{i = 1}^k \( \E Q_i^2 + \E (Q - Q_i)Q_i \) \le k \cdot \frac{2n^2 + n + 1 + nk}{2(n^2 - 1)^2} \\
    &\le \frac{k(3n^2 + n + 1)}{2(n^2 - 1)^2} \le \frac{k(4n^2 - 4)}{2(n^2 - 1)^2} = \frac{2k}{n^2 - 1}, \\
  \end{align*}
  which proves \eqref{eq:Q^2-bound}.
\end{proof}

\subsection{Completing the proof}

Proving Lemma \ref{lem:increment} is now a matter of assembling
together all of the bounds we have established.

\begin{proof}[Proof of Lemma \ref{lem:increment}]
  By Proposition \ref{prop:talagrand} and Corollary
  \ref{cor:L^2-formula}, we have
  \begin{align*}
    \W_2\( Z_1, Z_{1 - 1/n} + Y \)^2 &\le 2 \sum_{i = 1}^k \sigma_i^2 \( \E f(Z)^2 - \E f_{(i)}(Z)^2 \) \\
    &\le 2 \sum_{i = 1}^k \sigma_i^2 \E \( e^Q - e^{Q - Q_i} \).
  \end{align*}
  Thus, it remains to estimate the terms $\sigma_i^2 \E\( e^Q - e^{Q -
    Q_i} \)$. We do this by Taylor expanding the exponential. Define
  \[ R(t) = e^t - 1 - t - \frac{1}{2}t^2, \]
  so that
  \[ e^t = 1 + t + \frac{1}{2}t^2 + R(t). \]
  By Lemma \ref{lem:Q-estimates}, we can estimate the first and second
  order terms
  \begin{align*}
    \E (Q - (Q - Q_i)) &= \E Q_i \le -\frac{1}{2(n^2 - 1)} \\
    \frac{1}{2} \E (Q^2 - (Q - Q_i)^2) &= \E \( \frac{1}{2} Q_i^2 + (Q - Q_i)Q_i \) \le \frac{n^2 + n/2 + 1/2}{2(n^2 - 1)^2} + \frac{nk}{2(n^2 - 1)^2} \\
    &\le \frac{(n^2 + nk - 1) + nk}{2(n^2 - 1)^2} = \frac{1}{2(n^2 - 1)} + \frac{nk}{(n^2 - 1)^2}. \\
  \end{align*}
  To estimate the remainder term $R(Q) - R(Q - Q_i)$, note
  that for any $a, b \in [-1, 1]$,
  \begin{align*}
    |R(a) - R(b)| &= \left| (a - b) \sum_{m = 3}^\infty \frac{1}{m!} (a^{m - 1} + ab^{m - 2} + \ldots + b^{m - 1}) \right| \\
    &\le |a - b| \cdot \sum_{m = 3}^\infty \frac{1}{(m - 1)!} \( \frac{a^2 + b^2}{2} \) \le |a - b| \cdot \( \frac{a^2 + b^2}{2} \) \\
    &\le |a - b| \cdot \( \frac{a^2 + b^2 + (2a - b)^2}{2} \) = |a - b| \cdot \( \frac{3}{2} a^2 + (a - b)^2 \). \\
  \end{align*}
  In particular, by Lemma \ref{lem:|Q|-estimates}, both $Q$ and $Q -
  Q_i$ are in $[-1, 1]$, so
  \begin{align*}
    \E \left[ R(Q) - R(Q - Q_i) \right] &\le \E \left[ |Q_i| \( \frac{3}{2} Q^2 + Q_i^2 \) \right] \\
    &\le \frac{3}{2} \E |Q_i|Q^2 + \E \left[ \( \frac{n |Y_iY'_i|^2}{\sigma_i^2(n^2 - 1)} + \frac{1}{2n} \) Q_i^2 \right] \\
    &\le \frac{3}{2} \E |Q_i|Q^2 + \( \frac{n \beta^2}{\sigma_i^2(n^2 - 1)} + \frac{1}{2n} \) \cdot \frac{2n^2 + n + 1}{2(n^2 - 1)^2} \\
    &\le \frac{3}{2} \E |Q_i|Q^2 + \frac{2 \beta^2}{\sigma_i^2 n^3} + \frac{1}{n^3}.
  \end{align*}
  Thus,
  \begin{align*}
    \sigma_i^2 \E \( e^Q - e^{Q - Q_i} \) &= \sigma_i^2 \bigg( \E (Q - (Q - Q_i)) + \frac{1}{2} \E (Q^2 - (Q - Q_i)^2) \\
    &\hphantom{= \sigma_i^2 \bigg(\;} + \E(R(Q) - R(Q - Q_i)) \bigg) \\
    &\le \sigma_i^2 \bigg( - \frac{1}{2(n^2 - 1)} + \frac{1}{2(n^2 - 1)} + \frac{nk}{(n^2 - 1)^2} \\
    &\hphantom{= \sigma_i^2 \bigg(\;} + \frac{3}{2} \E |Q_i| Q^2 + \frac{2 \beta^2}{\sigma_i^2n^3} + \frac{1}{n^3} \bigg) \\
    &\le \frac{3}{2} \E \Big( \sigma_i^2 |Q_i| Q^2 \Big) + \frac{2 \beta^2}{n^3} + \frac{nk\sigma_i^2}{(n^2 - 1)^2} + \frac{\sigma_i^2}{n^3}. \\
    &\le \frac{3}{2} \E \Big( \sigma_i^2 |Q_i| Q^2 \Big) + \frac{2 \beta^2}{n^3} + \frac{3k \sigma_i^2}{n^3}. \\
  \end{align*}
  Summing over all $i$, we have
  \begin{align*}
    \sum_{i = 1}^k \sigma_i^2 \E \( e^Q - e^{Q - Q_i} \) &\le \frac{3}{2} \E \( Q^2 \sum_{i = 1}^k \sigma_i^2 |Q_i| \) + \frac{2 k \beta^2}{n^3} + \frac{3k}{n^3} \sum_{i = 1}^k \sigma_i^2 \\
    &\le \frac{3}{2} \E \( Q^2 \sum_{i = 1}^k \sigma_i^2 |Q_i| \) + \frac{5k \beta^2}{n^3} \\
    &\le \frac{3}{2} \E \( Q^2 \sum_{i = 1}^k \sigma_i^2 \( \frac{n^2 |Y_iY'_i|}{\sigma_i^2(n^2 - 1)} + \frac{1}{2n} \) \) + \frac{5k \beta^2}{n^3} \\
    &\le \frac{3}{2} \E Q^2 \( \frac{n \beta^2}{n^2 - 1} + \frac{\beta^2}{2n} \) + \frac{5k \beta^2}{n^3} \\
    &\le \frac{3k}{n^2 - 1} \cdot \frac{2n \beta^2}{n^2 - 1} + \frac{5k \beta^2}{n^3} \le \frac{25k \beta^2}{2n^3}, \\
  \end{align*}
  and so
  \[ \W_2\( Z_1, Z_{1 - 1/n} + Y \)^2 \le 2 \sum_{i = 1}^k \sigma_i^2 \E \( e^Q - e^{Q - Q_i} \) \le \frac{25k \beta^2}{n^3}, \]
  as desired.
\end{proof}

\section{Appendix} \label{sec:appendix}

\subsection{Proof of Proposition \ref{prop:lower-bound}} \label{subsec:lower-bound-proof}
\begin{proof}
  Let $\ell_n = \frac{\beta}{\sqrt{n}}$, and consider the lattice $L =
  \ell_n \Z^d$. For any $x \in \R^d$, let $d_L(x)$ denote the minimum
  Euclidean distance from $x$ to $L$. Note that $S_n$ takes values in
  $L$. Thus, letting $\rho$ denote the density of $Z$, we have
  \[ \W_2(S_n, Z) \ge \int \rho(x) d_L(x) \,dx. \]

  To estimate the right hand side, for any $y \in L$, let $Q_n(y)$
  denote the cube of side length $\ell_n$ centered at $y$ (which is
  also the set of points in $\R^d$ closer to $y$ than to any other
  point in $L$). We find that
  \begin{align}
    \frac{1}{\Vol Q_n(y)} \int_{Q_n(y)} d_L(x) \,dx &= \frac{\ell_n}{2^d} \int_{[-1,1]^d} \|x\| \,dx = \frac{\ell_n}{2^d} \int_{[-1,1]^d} \sqrt{x_1^2 + \cdots + x_d^2} \,dx \nonumber \\
    &\ge \frac{\ell_n}{2^d} \int_{[-1,1]^d} \frac{1}{\sqrt{d}}\(|x_1| + \cdots + |x_d|\) \,dx = \frac{1}{2} \ell_n \sqrt{d}. \label{eq:cube-dist}
  \end{align}

  Next, let $M$ be large enough so that,
  \[ \int_{[-M,M]^d} \rho(x) \,dx \ge \frac{1}{2}, \]
  and let
  \begin{equation}
    r_n = \inf_{\substack{x, y \in [-2M, 2M]^d \\ \|x - y\| \le \sqrt{d} \ell_n}} \frac{\rho(x)}{\rho(y)}. \label{eq:r_n-def}
  \end{equation}
  Note that since $\rho$ is positive and continuous, we have $\lim_{n
    \rightarrow \infty} r_n = 1$.

  Assume now that $n$ is sufficiently large so that $\ell_n <
  M$. Combining \eqref{eq:r_n-def} with \eqref{eq:cube-dist}, we have
  for each $y \in L \cap [-M,M]^d$ that
  \begin{align*}
    \int_{Q_n(y)} \rho(x) d_L(x) \,dx &\ge \frac{r_n}{\Vol Q_n(y)} \int_{Q_n(y)} \rho(x) \,dx \cdot \int_{Q_n(y)} d_L(x) \,dx \\
    &\ge \frac{r_n \ell_n \sqrt{d}}{2} \int_{Q_n(y)} \rho(x) \,dx.
  \end{align*}
  Summing over all such $y$ yields
  \begin{align*}
    \W_2(S_n, Z) &\ge \int \rho(x) d_L(x) \,dx \ge \int_{[-2M,2M]^d} \rho(x) d_L(x) \,dx \\
    &\ge \sum_{y \in L \cap [-M,M]^d} \int_{Q_n(y)} \rho(x) d_L(x) \,dx \\
    &\ge \frac{r_n \ell_n \sqrt{d}}{2} \int_{[-M,M]^d} \rho(x) \,dx \ge \frac{r_n \beta \sqrt{d}}{4 \sqrt{n}}.
  \end{align*}
  Multiplying both sides by $\sqrt{n}$ and taking limits gives the
  result.
\end{proof}

\subsection{Proof of Proposition \ref{prop:w2-to-convex}} \label{subsec:w2-to-convex-proof}
\begin{proof}
  We prove the result with $C = 5$. Let $A \subset \R^d$ be a given
  convex set. For a parameter $\epsilon$ to be specified later, define
  \[ A^\epsilon = \{ x \in \R^d \mid \sup_{a \in A} \|x - a\| \le \epsilon \} \]
  \[ A_\epsilon = \{ x \in \R^d \mid \inf_{a \in \R^d \setminus A} \|x - a\| \ge \epsilon \}. \]
  Ball \cite{B93} showed a $4d^{1/4}$ upper bound\footnote{The
    constant was later improved to $(2 \pi)^{-1/4} \approx 0.64$ by
    Nazarov \cite{N03}, who also constructed an example with surface
    area of order $d^{1/4}$.} for the Gaussian surface area of any
  convex set in $\R^d$. Hence,\footnote{This is also given as equation
    (1.4) in \cite{B03}.}
  \[ \P\(Z \in A^\epsilon \setminus A\) \le 4 \epsilon d^{1/4}, \text{ and } \P\(Z \in A \setminus A_\epsilon\) \le 4 \epsilon d^{1/4}. \]
  We may regard $T$ as being coupled to $Z$ so that $\E \|T - Z\|^2 =
  \W_2(T, Z)^2$. Then,
  \begin{align*}
    \P(T \in A) &\le \P(\|T - Z\| \le \epsilon, \; T \in A) + \P(\|T - Z\| > \epsilon) \\
    &\le \P(Z \in A^\epsilon) + \epsilon^{-2} \W_2(T, Z)^2 \\
    &\le \P(Z \in A) + 4 \epsilon d^{1/4} + \epsilon^{-2} \W_2(T, Z)^2
  \end{align*}
  Similarly,
  \begin{align*}
    \P(Z \in A) &\le \P(Z \in A_\epsilon) + 4 \epsilon d^{1/4} \\
    &\le \P(\|T - Z\| \le \epsilon, \; Z \in A_\epsilon) + \P(\|T - Z\| > \epsilon) + 4 \epsilon d^{1/4} \\
    &\le \P(T \in A) + \epsilon^{-2}\W_2(T, Z)^2 + 4 \epsilon d^{1/4}.
  \end{align*}
  Thus,
  \[ |\P(T \in A) - \P(Z \in A)| \le \epsilon^{-2}\W_2(T, Z)^2 + 4 \epsilon d^{1/4}, \]
  and taking $\epsilon = d^{-1/12} \W_2(T, Z)^{2/3}$ gives
  the result.
\end{proof}

\subsection{Proof of Lemma \ref{lem:conditional-L^2-lemma}} \label{subsec:conditional-L^2-lemma}
\begin{proof}
  Let $A'$ and $B'$ be independent copies of $A$ and $B$. Then,
  \[ \E \Big( f(A, B) + f(A', B') - f(A, B') - f(A', B) \Big)^2 \ge 0. \]
  Expanding yields
  \begin{align*}
    4 \E f(A, B)^2 + 4 (\E f(A, B))^2 &= 4 \E f(A, B)^2 + 2 \E f(A, B) f(A', B') \\
    &\hphantom{==} + 2 \E f(A, B') f(A', B) \\
    &\ge 2 \E f(A, B)f(A, B') + 2 \E f(A, B) f(A', B) \\
    &\hphantom{==} + 2 \E f(A', B') f(A, B') + 2 \E f(A', B') f(A', B) \\
    &= 2 \E f_B(A)^2 + 2 \E f_A(B)^2 \\
    &\hphantom{==} + 2 \E f_A(B)^2 + 2 \E f_B(A)^2 \\
    &= 4 \E f_B(A)^2 + 4 \E f_A(B)^2,
  \end{align*}
  as desired.
\end{proof}

\subsection{Proof of Equation \eqref{eq:entropy-bound}} \label{subsec:entropy-bound-proof}
\begin{proof}
  We proceed by induction on the dimension $d$, retracing the argument
  of \cite{T96}, \S 3. The base case $d = 1$ is immediate from Theorem
  \ref{thm:talagrand-orig}.

  Assume now that the inequality holds in $d - 1$ dimensions. For the
  inductive step, we can follow the same argument used to prove
  Theorem \ref{thm:talagrand-orig} (see \cite{T96}, \S 3). The
  argument proceeds by first comparing $Y$ to another $\R^d$-valued
  random variable $\hat{Y}$ sharing the first $d - 1$ coordinates of
  $Y$, but whose last coordinate is independently drawn from
  $\mathcal{N}(0, \sigma_d)$.

  Fix a $(d - 1)$-dimensional vector $\hat{x}$, and let $T_{\hat{x}}$
  denote a random variable distributed as the last coordinate of $Y$
  conditioned on the first $d - 1$ coordinates being equal to
  $\hat{x}$. Let $\hat{\rho}(\hat{x}) = \int_{-\infty}^\infty
  \rho(\hat{x}, t) dt$. Then, the density of $T_{\hat{x}}$ at $t$ is
  given by
  \[ \frac{f(\hat{x}, t) \cdot \rho(\hat{x}, t)}{f_{(d)}(\hat{x}, 0) \cdot \hat{\rho}(\hat{x})}. \]
  Noting that $\frac{\rho(\hat{x}, t)}{\hat{\rho}(\hat{x})}$ is the
  density of $\mathcal{N}(0, \sigma_d)$ at $t$, the one-dimensional
  case of Theorem \ref{thm:talagrand-orig} implies
  \begin{equation} \label{eq:xhat-w2}
    \W_2(T_{\hat{x}}, \mathcal{N}(0, \sigma_d))^2 \le 2 \sigma_d^2 \int_{-\infty}^\infty \frac{f(\hat{x}, t)}{f_{(d)}(\hat{x}, t)} \log \( \frac{f(\hat{x}, t)}{f_{(d)}(\hat{x}, t)} \) \frac{\rho(\hat{x}, t)}{\hat{\rho}(\hat{x})} \,dt.
  \end{equation}
  Since $T_{\hat{x}}$ and $\mathcal{N}(0, \sigma_d)$ have the same
  distributions as $Y$ and $\hat{Y}$ conditioned on $\hat{x}$, we may
  integrate \eqref{eq:xhat-w2} over $\hat{x}$ to obtain
  \begin{align*}
    \W_2(Y, \hat{Y})^2 &\le 2 \int_{\R^{d - 1}} \W_2(T_{\hat{x}}, \mathcal{N}(0, \sigma_d))^2 \cdot f_{(d)}(\hat{x}, 0) \hat{\rho}(\hat{x}) \,d\hat{x} \\
    &\le 2 \sigma_d^2 \int_{\R^{d - 1}} \int_{-\infty}^\infty f(\hat{x}, t) \log \( \frac{f(\hat{x}, t)}{f_{(d)}(\hat{x}, t)} \) \rho(\hat{x}, t) \,dt \,d\hat{x}. \\
    & = 2 \sigma_d^2 \cdot \E \( f(Z) \log \frac{f(Z)}{f_{(d)}(Z)} \) \\
    & = 2 \sigma_d^2 \cdot \bigg( \E \( f(Z) \log f(Z) \) - \E \( f_{(d)}(Z) \log f_{(d)}(Z) \) \bigg)
  \end{align*}

  Next, define $Y_{(d)}$ and $Z_{(d)}$ to be the projections onto the
  first $d - 1$ coordinates of $Y$ and $Z$, respectively. Note that
  the coupling of $Y$ to $\hat{Y}$ changes only $d$-th
  coordinate. Furthermore, the $d$-th coordinates of $\hat{Y}$ and $Z$
  are both distributed as $\mathcal{N}(0, \sigma_d)$ independent of
  the first $d - 1$ coordinates. Thus, a coupling of $Y_{(d)}$ to
  $Z_{(d)}$ induces a coupling of $\hat{Y}$ to $Z$ in which the last
  coordinate does not change. Consequently,
  \begin{equation} \label{eq:entropy-bound-induction-step}
    \W_2(Y, Z)^2 \le 2 \sigma_d^2 \cdot \bigg( \E \( f(Z) \log f(Z) \) - \E \( f_{(d)}(Z) \log f_{(d)}(Z) \) \bigg) + \W_2(Y_{(d)}, Z_{(d)})^2.
  \end{equation}
  Now, recall that the density of $Y_{(d)}$ at a point $\hat{x} \in
  \R^{d - 1}$ is $f_{(d)}(\hat{x}, 0) \cdot \hat{\rho}(\hat{x})$, and
  so applying the inductive hypothesis to $\W_2(Y_{(d)}, Z_{(d)})^2$
  yields
  \[ \W_2(Y_{(d)}, Z_{(d)})^2 \le 2 \sum_{k = 1}^{d - 1} \sigma_k^2 \cdot \E \( f_{[k]}(Z_{(d)}) \log f_{[k]}(Z_{(d)}) - f_{[k - 1]}(Z_{(d)}) \log f_{[k - 1]}(Z_{(d)}) \) \]
  \[ = 2 \sum_{k = 1}^{d - 1} \sigma_k^2 \cdot \E \( f_{[k]}(Z) \log f_{[k]}(Z) - f_{[k - 1]}(Z) \log f_{[k - 1]}(Z) \). \]
  Substituting into \eqref{eq:entropy-bound-induction-step}, we obtain
  \[ \W_2(Y, Z)^2 \le 2 \sum_{k = 1}^{d - 1} \sigma_k^2 \cdot \E \( f_{[k]}(Z) \log f_{[k]}(Z) - f_{[k - 1]}(Z) \log f_{[k - 1]}(Z) \), \]
  completing the induction.
\end{proof}

\subsection{Proof of Lemma \ref{lem:gaussian-expectation}} \label{subsec:gaussian-expectation-proof}
\begin{proof}
  Let $C_k = (2 \pi)^{-\frac{k}{2}}$. We have

  \small
  \[ \E \left[ \exp\( a\|Z\|_\Siginv^2 + b\langle Z, v \rangle_\Siginv \) \right] = \frac{C_k}{\sqrt{\det \Sigma}} \bigintsss_{\R^k} \exp\(- \( \frac{1}{2} - a \) \|x\|_\Siginv^2 + b \langle x, v \rangle_\Siginv \) dx \]
  \begin{eqnarray*}
  &=& \frac{C_k}{\sqrt{\det \Sigma}} \bigintsss_{\R^k} \exp\(- \( \frac{1}{2} - a \) \left\| x - \( \frac{b}{1 - 2a} \) v \right\|_\Siginv^2 + \frac{b^2}{2 - 4a}\|v\|_\Siginv^2 \) dx \\
  &=& \exp\( \frac{b^2}{2 - 4a} \|v\|_\Siginv^2 \) \cdot \frac{C_k}{\sqrt{\det \Sigma}} \bigintsss_{\R^k} \exp\(- \frac{1 - 2a}{2} \left\| x - \( \frac{b}{1 - 2a} \) v \right\|_\Siginv^2 \) dx \\
  &=& \exp\( \frac{b^2}{2 - 4a} \|v\|_\Siginv^2 \) \cdot \frac{C_k}{\sqrt{\det \Sigma}} \bigintsss_{\R^k} \exp\(- \frac{1 - 2a}{2} \left\| x  \right\|_\Siginv^2 \) dx \\
  &=& \exp\( \frac{b^2}{2 - 4a} \|v\|_\Siginv^2 \) \cdot \frac{C_k}{\sqrt{\det \Sigma}} \bigintsss_{\R^k} \exp\(- \frac{1}{2} \left\| y \right\|_\Siginv^2 \) \frac{dy}{\sqrt{(1 - 2a)^k}} \quad (y := \sqrt{1 - 2a} \cdot x) \\
  &=& \exp\( \frac{b^2}{2 - 4a} \|v\|_\Siginv^2 \) \cdot \( \frac{1}{1 - 2a} \)^{k/2} \cdot \frac{C_k}{\sqrt{\det \Sigma}} \bigintsss_{\R^k} \exp\(- \frac{1}{2} \left\| y \right\|_\Siginv^2 \) dy \\
  &=& \exp\( \frac{b^2}{2 - 4a} \|v\|_\Siginv^2 \) \cdot \( \frac{1}{1 - 2a} \)^{k/2}. \\
  \end{eqnarray*}
  \normalsize
\end{proof}

\end{document}